\documentclass[12pt, a4paper, reqno]{amsart}
\usepackage{amsmath, amsthm, amscd, amsfonts, amssymb, graphicx, xcolor}
\usepackage{tikz}
\usetikzlibrary{cd}
\usepackage[bookmarksnumbered, plainpages, hidelinks]{hyperref}
\usepackage[margin=1in]{geometry}
\usepackage{verbatim}
\usepackage{mathtools}
\usepackage{mathabx}
\usepackage{mathrsfs}

\newtheorem{theorem}{Theorem}[section]
\newtheorem{lemma}[theorem]{Lemma}
\newtheorem{prop}[theorem]{Proposition}
\newtheorem{cor}[theorem]{Corollary}
\newtheorem{fact}[theorem]{Fact}
\theoremstyle{definition}
\newtheorem{definition}[theorem]{Definition}
\newtheorem{example}[theorem]{Example}
\newtheorem{assumption}[theorem]{Assumption}

\newtheorem{remark}[theorem]{Remark}
\newtheorem{notation}[theorem]{Notation}
\numberwithin{equation}{section}

\def\Ind#1#2{#1\setbox0=\hbox{$#1x$}\kern\wd0\hbox to 0pt{\hss$#1\mid$\hss}
\lower.9\ht0\hbox to 0pt{\hss$#1\smile$\hss}\kern\wd0}

\def\ind{\mathop{\mathpalette\Ind{}}}

\DeclareMathOperator{\locus}{locus}

\DeclareMathOperator{\acl}{acl} \DeclareMathOperator{\dcl}{dcl} \DeclareMathOperator{\scf}{SCF}
\DeclareMathOperator{\Alg}{Alg}
\DeclareMathOperator{\alg}{alg}
\DeclareMathOperator{\id}{id}
\DeclareMathOperator{\fr}{Fr}
\DeclareMathOperator{\aut}{Aut}
\DeclareMathOperator{\tp}{tp}

\newcommand{\be}{\mathcal{B}}
\newcommand{\ce}{\mathscr{K}}

\newcommand{\ec}{\ce\be_\varphi\operatorname{CF}}

\DeclareMathOperator{\spec}{Spec}
\newcommand{\ga}{\mathbb{G}_{\rm{a}}}

\newcommand{\ka}{k}

\newcommand{\ra}{\longrightarrow}

\newcommand{\bcf}{\mathcal{B}\operatorname{CF}}

\newcommand{\bv}{\be_\varphi}
\newcommand{\bvcf}{\bv\operatorname{CF}}
\newcommand{\bpac}{\bv\operatorname{PAC}}
\newcommand{\neat}{\operatorname{Fr}\left( \ker \pi_\be \right) = 0}

\begin{document}

\title{Existentially closed fields with operators in various categories}

\author[J. Gogolok]{Jakub Gogolok}
\thanks{Supported by the Narodowe Centrum Nauki grant no. 2023/49/N/ST1/02512.}
\address{Instytut Matematyczny, Uniwersytet Wrocławski, Wrocław, Poland}
\email{jakub.gogolok@math.uni.wroc.pl}
\urladdr{http://www.math.uni.wroc.pl/\textasciitilde gogolok/}



\begin{abstract}
We deal with fields with certain operators - introduced by the author and Kowalski - which we call $\be$-fields. We are mostly interested in $\be$-fields which are existentially closed, possibly in some restricted category of $\be$-fields. This in particular involves seeking for a model companion, but also is related to so-called pseudo algebraically closed structures. We prove a very general result saying that in many cases being existentially closed (in a generalized sense) is an elementary property. This encompasses, generalizes and simplifies many results from the literature. We study the resulting first-order theories, most importantly we study dividing lines and quantifier elimination.
\end{abstract} \maketitle

\section{Introduction}

This note is about fields with operators, through the point of view of algebra and model theory. The two prototypical examples of fields with operators are differential fields (i.e. fields with a distinguished derivation) and difference fields (i. e fields with a distinguished endomorphism). These classes are immensely interesting, on one hand because they have many interesting model-theoretic properties and on the other hand because of their applications. Just to name a few applications:
\begin{enumerate}
    \item Various results in algebraic dynamics (see \cite{algDim}),
    \item Hrushovski's proof of the Manin-Mumford conjecture (see \cite{udiManin}),
    \item Ax-Lindemann-Weierstrass theorems for uniformizers of Fuchsian groups (see \cite{axLinde}).
\end{enumerate}

Both differential and difference fields are  instances of $\mathcal{D}$\textit{-rings structures} introduced by Moosa and Scanlon in \cite{MS1}  (an equivalent set-up of $B$\textit{-operators} was given in \cite{BHKK} by Beyarslan, Hoffmann, Kamensky and Kowalski). What this means is that both these classes are governed by a certain $\ka$-algebra in the following manner. Fix some base ring $\ka$. For a $\ka$-algebra $R$, let $R\left[ \varepsilon \right]$ be the ring of dual numbers over $R$, i.e. $R\left[ \varepsilon \right] = R\left[ X \right] / \left( X^2\right)$ and $\varepsilon$ is the coset of $X$. Then, the map $\partial\colon R \to R$ on a $\ka$-algebra $R$ is a $\ka$-derivation (i.e. a derivation vanishing on $\ka$) if and only if the map
$$R\to R\left[ \varepsilon \right], \ x\mapsto x+\partial(x)\varepsilon,$$
is a morphism of $\ka$-algebras. Since $R\left[ \varepsilon \right]=R\otimes_\ka \ka\left[ \varepsilon \right]$, we may say that ($\ka$-)derivations are governed by the algebra $\ka\left[ \varepsilon \right]$. Similarly, endomorphisms (fixing $\ka$) are governed by the algebra $\ka \times \ka$.


Certain natural examples of operators, e.g. derivations of the Frobenius map (see \cite{DerFro} or \cite{DerFro2}) do not fit in the above setup. Those can be described using \textit{ring schemes} leading to the notion of $\be$-operators developed by the author and Kowalski in \cite{beta-op}. 

In this paper we further develop the theory of $\be$-fields (i.e. fields with a $\be$-operator). In particular, we study the geometric notion of \textit{$\be$-varieties} (see Section \ref{sec: bvar}) and investigate \textit{iterative $\be$-fields} (see Section \ref{sec: iterative}). Having achieved that, the rest of the paper is devoted to the analysis of $\be$-fields which are existentially closed, possibly in a generalized sense e. g. only in regular extensions. The fundamental question is whether being existentially closed (in a generalized sense) is an elementary property. This question belongs to a well-established line of research in model theory and in particular entails the pursuit of model companions in algebraic model theory. We prove a very general result in this direction (see Section \ref{sec: general}, Theorem \ref{mainThm}) which entails and simplifies many results from the literature (see Remark \ref{remark: what we did}). We give many applications of this result (see Section \ref{sec: apl}) and also analyze some model-theoretic properties of the resulting objects.

\subsection*{Acknowledgment} This work is based on my PhD thesis written at the University of Wrocław, therefore I would like to thank my advisor Piotr Kowalski for his guidance throughout my PhD studies, numerous fruitful discussions and his patience. I thank the Wrocław model theory group for various comments given during the seminar talks I gave at the University of Wrocław. I would also like to thank the referees of my PhD thesis Moshe Kamensky, Rahim Moosa and Omar Le\'{o}n S\'{a}nchez for all their comments, suggestions and discussion. All those people significantly contributed to improving the work presented in this paper.

\section{\texorpdfstring{Preliminaries on $\mathcal{B}$-operators}{Preliminaries on B-operators}}
\label{chapter: B-op}
For the rest of the paper we fix a base field $\ka$ (although for some of the results it is enough to assume that $\ka$ is a ring). Let us give a quick overview of the theory of $\be$-fields presented in \cite{beta-op}. Some definitions differ from the ones given in \cite{beta-op} since we 

We start from the notion of an affine ring scheme over $\ka$, which is completely analogous the the notion of an affine group scheme.
\begin{definition}
An \textbf{affine ring scheme over $\ka$} is a representable functor from the category of $\ka$-algebras to the category of rings.
\end{definition}

\begin{example}
The affine line, represented by the polynomial ring $\ka \left[ X \right]$, is a ring scheme over $\ka$. As a functor it is simply the forgetful functor from the category of $\ka$-algebras to the category of rings. We will denote this ring scheme by $\mathbb{S}_{\ka}$.
\end{example}

\begin{definition}
A \textbf{$\ka$-algebra scheme} is a ring scheme $\mathcal{B}$ together with a morphism $\iota:\mathbb{S}_{\ka}\ra \mathcal{B}$ of ring schemes over $\ka$. As in the case of $\ka$-algebras, we often say simply ``$\mathcal{B}$ is a $\ka$-algebra scheme'', suppressing $\iota$ from the notation.
\end{definition}

\begin{remark}\label{ralg}
Let $\mathcal{B}$ be a $\ka$-algebra scheme and let $R$ be a $\ka$-algebra. The structure map evaluated on $R$:
$$\iota_R:\mathbb{S}_{\ka}(R)=R\ra \mathcal{B}(R)$$
gives $\mathcal{B}(R)$ the structure of an $R$-algebra.
\end{remark}

\begin{definition}\label{bdef}
A \textbf{coordinate $\ka$-algebra scheme} is a triple $(\mathcal{B},\iota, \pi )$ where:
\begin{enumerate}
\item $(\mathcal{B},+)=\ga^e$,
\item $(\mathcal{B},\iota )$ is a $\ka$-algebra scheme,
\item $\pi:\mathcal{B}\ra \mathbb{S}_{\ka}$ is a morphism of $\ka$-algebra schemes (that is: a morphism of ring schemes satisfying $\pi\circ\iota=\id$) such that under the above identification $(\mathcal{B},+)=\ga^e$, the morphism $\pi$ is the projection on the first coordinate.
\end{enumerate}
\end{definition}

For everything what follows we fix a coordinate $\ka$-algebra scheme $(\mathcal{B},\iota, \pi )$ with $\left( \be, + \right) = \ga^e$.

\begin{definition}\label{bopdef}
Let $R$ and $S$ be $\ka$-algebras. A \textbf{$\mathcal{B}$-operator on $R$} is a $\ka$-algebra map $\partial\colon R\to \mathcal{B}(R)$ such that $\pi_R\circ \partial=\id$. More generally, given a $\ka$-algebra map $f\colon R\to S$ a \textbf{$\mathcal{B}$-operator from $R$ to $S$ (of $f$)} is a $\ka$-algebra map $\partial \colon R\to \mathcal{B}(S)$ such that $\pi_S\circ \partial=f$. The \textbf{ring of constants} of such a $\partial\colon R\to \mathcal{B}(S)$ is defined as:
$$R^{\partial}:=\{r\in R\ |\ \partial(r)=\iota_S\left( f \left( r \right) \right) \}.$$
\end{definition}

\begin{definition}
A field with a $\mathcal{B}$-operator is called a \textbf{$\mathcal{B}$-field}. Similarly, we define \textbf{$\mathcal{B}$-rings}, \textbf{$\mathcal{B}$-field extensions}, etc.    
\end{definition}

\begin{remark}
\label{remark: beta-op = twisted B-op.}\label{mscontext}
Let us unwind the definition of a $\be$-operator. Since $\left( \be, + \right) = \ga^e$, we see that $\be$-operators from $R$ to $S$ correspond to certain sequences of maps $\partial = \left( \partial_1,\dotsc , \partial_e\right)$ from $R$ to $S$. More precisely, there are polynomials $F_1,\dotsc , F_e\in \ka \left[ X_1,Y_1,\dotsc , X_e, Y_e\right]$ such that $\partial$ as above is a $\be$-operator if and only if for each $i$, the map $\partial_i$ is additive and satisfies for any $x,y\in R$ the following identity:
$$\partial_i(xy)=F_i\left(\partial_1(x),\partial_1(y),\ldots,\partial_e(x),\partial_e(y)\right).$$
Of course, the polynomials $F_1,\dotsc , F_e$ are just the polynomials defining the multiplication law on $\be$.
\end{remark}

\section{\texorpdfstring{$\be$}{B}-varieties}\label{sec: bvar}
Let us work in a big algebraically closed field $\Omega$. Suppose $V\subseteq \Omega^n$ is an affine\footnote{Since we work in the affine realm, from now on we drop the adjective ``affine''} variety (or even just an algebraic set) over $\left( K, \partial \right)$. We would like be able to speak about point of the form $\partial\left( a \right) \in \Omega^{ne}$ for $a\in V\left( K \right)$. This would allow us to nicely code $\partial$-equations (i.e. equations involving $\partial$ and its iterations) using varieties, as any $\partial$-equation (for now involving no iterations of $\partial$) in a tuple of variables $X$ can be phrased as ``$\partial\left( X \right) \in W$'' for an appropriate algebraic set $W$.\footnote{Concretely, if the equation is $F\left( X, \partial_1\left( X \right), \partial_2\left( X \right), \dotsc \right) = 0 $ for a polynomial $F\left( X_0, X_1, X_2, \dotsc \right)$, then $W$ is simply the zero set of $F$.} Thus it would be nice to have a $K$-variety (or just an algebraic set) $\tau^\partial V$ such that for any $\be$-extension $K\subseteq L$ and any $a\in V\left( L \right)$ we have $\partial\left( a \right) \in \tau^\partial V\left( L \right)$, and moreover $\tau^\partial V$ is the smallest algebraic set with this property. From this description it is easy to guess what $\tau^\partial V$ should be. Namely, if $V$ is given by an ideal $I\trianglelefteq K\left[ X \right]$, then $\tau^\partial V$ is the zero set in $\Omega^{ne}$ of the ideal $\overline{I}:=\left( \partial_1\left( I \right)\cup \dotsc \cup \partial_e\left( I \right) \right) \trianglelefteq K\left[ \overline{X} \right]$ where $\overline{X} = \left( X_1,\dotsc, X_e \right)$. An important point is that $\tau^\partial$ may be defined as a certain adjoint functor (see \cite{beta-op} for details).

We want to adapt some classical geometric notions from algebraic and differential-algebraic geometry to the context of $\be$-algebra. Fix a $\be$-field $\left( K, \partial \right)$.
\begin{definition}
A \textbf{$\be$-variety over $\left( K, \partial \right)$} is a pair $\left( V, s \right)$ consisting of an affine $K$-variety $V$ together with a morphism $s\colon V \to \tau^\partial V$ which is a section of $\pi^\partial_V \colon \tau^\partial V \to V$, i.e. it satisfies $\pi^\partial_V \circ s = \id_V$.
\end{definition}
Since most of the time we will work only with varieties over our fixed $\be$-field $\left( K, \partial \right)$, we will often say ``$\be$-variety'' instead of ``$\be$-variety over $\left( K, \partial \right)$'', if no confusion arises. Also, we will say that $\left( V, s \right)$ is absolutely irreducible (resp. separable) if the underlying $K$-variety $V$ is such.

\begin{remark}\label{remark: B-varieties}
Let $\left( V, s \right)$ be a $\be$-variety. Then $s$ correspond to a $K$-algebra homomorphism $s^*\colon K\left[ \tau^\partial V \right] \to K\left[ V \right]$, which in turn corresponds by adjointness to a $K$-algebra homomorphism $\partial_s\colon K\left[ V \right] \to \be^\partial\left( K\left[ V \right] \right)$. One easily checks that under this correspondence the condition $\pi^\partial_V \circ s = \id_V$ translates into $\pi_{K\left[ V \right]} \circ \partial_s = \id_{ K\left[ V \right]}$, i.e. that $\partial_s$ is a $\be$-operator on $K\left[ V \right]$ extending $\partial$ on $K$. Thus a $\be$-variety structure on $V$ is the same as a $\be$-ring (over $\left( K ,\partial \right)$) structure on $K\left[ V \right]$.
\end{remark}

\begin{remark}
Remark \ref{remark: B-varieties} can be elegantly phrased as: the category of finitely generated (in the algebraic sense) $\be$-domains over $\left( K, \partial \right)$ is dual to the naturally defined category of $\be$-varieties. This is completely analogous to the case of classical algebraic geometry. In the same vain we could define (affine) $\be$-schemes \textit{et cetera}. We do not do any of this, since we really only use $\be$-varieties as a simple way of encoding certain $\be$-algebras over $K$. We point out however that e. g. difference schemes where fundamental in Hrushovski's work on the elementary theory of the Frobenius automorphism (see \cite{HrFro}).
\end{remark}

The reason we need $\be$-varieties is because they can code ``systems of $\partial$-equation''. This is analogous to algebraic geometry where a variety codes a system of polynomial equations. In the classical setting a $K$-rational point of a variety $V$ is the same as a solution in $K$ of the system coded by $V$. The $\be$-algebraic counterpart of this notion is the following.
 
\begin{definition}\label{def: B-points}
Let $\left( V, s \right)$ be a $\be$-variety and let $L$ be a $\be$-extension of $K$. An \textbf{$L$-rational $\be$-point $\left( V, s \right)$} is a point $a\in V\left( L \right)$ such that $s\left( a \right) = \partial\left( a \right)$. We denote the set of all $L$-rational $\be$-points by $\left( V, s \right)^\sharp \left( L \right)$.
\end{definition}

Elaborating on Remark \ref{remark: B-varieties} yields the following useful lemma.

\begin{lemma}\label{generic point is sharp}
Let $\left( V, s \right)$ be a $\be$-variety and let $a$ be a generic point of $V$ over $K$. Under the identification $K\left( V \right) = K\left( a \right)$ we have that $a\in \left( V, s \right)^\sharp \left( K\left( V\right) \right)$.
\end{lemma}
\begin{proof}
Let $I$ be the ideal of $V$, write $K\left[ V \right] = K\left[ X \right] / I$ and identify $a$ with $X+I$. We have $K\left[ \tau^\partial V \right] = K\left[ \overline{X} \right] / \overline{I}$ where $\overline{X} = \left( X_1,\dotsc, X_e \right)$ and $\overline{I}=\left( \partial_1\left( I \right)\cup \dotsc \cup \partial_e\left( I \right) \right) \trianglelefteq K\left[ \overline{X} \right]$. The adjoint of $s^*$, i.e. $\partial_s$, is by definition the map $\partial_s\colon K\left[ \overline{X} \right] / \overline{I} \to \be\left( K\left[ X \right] / I \right)$ induced by sending $X$ to the tuple $\left( s^*\left( X_1 \right),\dotsc , s^*\left( X_e \right)\right)$, thus $\partial_s\left( X+I \right) = s^*\left( X_i \right) + \overline{I} = s\left( X+I \right)$, i.e. $X+I$ is a $\be$-point of $\left( V, s \right)$.
\end{proof}

\section{Iterative \texorpdfstring{$\be$}{B}-fields}\label{sec: iterative}

In this section we will first review various types of iterative operators appearing in the literature and how they generalize to $\be$-operators, see Subsection \ref{sec: iterative, sub: examples}. In Subsection \ref{sec: iterative, sub: algebra} we investigate some basic properties of iterative $\be$-operators. 

\subsection{Examples}\label{sec: iterative, sub: examples}

The most natural example of iterative operators are group actions on fields, which are of course extremely important in mathematics. Also, model-theoretic considerations of group actions on fields are very interesting and fruitful (most notably, the consideration of $\mathbb{Z}$-actions, i.e. the case of the theory $\operatorname{ACFA}$).

For a fixed finite group $G$, the theory of $G$-fields (i.e. fields with an action of $G$ by field automorphisms) is companionable, as proved by Hoffmann-Kowalski in \cite{HK1}. The datum of a $G$-field $K$, i.e. the sequence of automorphisms $\left( \sigma_g\colon K \to K | g\in G \right)$, forms a $\ka^G$-operator, where $\ka^G$ is equipped with a convolution product (alternatively, $\ka^G$ is the bialgebra dual to the group ring $\ka\left[ G \right]$). The iterativity condition (i.e. the statements $\sigma_{gh} = \sigma_g \circ \sigma_h$ for $g, h\in G$) can be expressed using the Hopf algebra structure of $\ka^G$, as we will see in the next paragraph.

More generally, we can consider actions of \textit{finite group schemes}. Given a finite (affine) group scheme $\mathfrak{g}$ over $\ka$ whose corresponding Hopf algebra is $H$, we can consider $\mathfrak{g}$-fields, i.e. fields $K$ with an action of $\mathfrak{g}$ on $\spec K$. This is the same as equipping $K$ with an $H$-operator $\partial\colon K\to K\otimes_\ka H$ (where $\pi\colon H\to \ka$ is the counit of $H$) such that the following diagram commutes
\vspace{0.5cm}
\begin{center}
\begin{tikzcd}[row sep=huge,column sep=huge,every label/.append
style={font=\small}]
R \arrow[r, "\partial"] \arrow[d, "\partial"'] & R\otimes_\ka H \arrow[d, "\partial \otimes \id"] \\
R\otimes_\ka H \arrow[r, "\id \otimes \mu"']   & R\otimes_\ka H \otimes_\ka H                    
\end{tikzcd}
\end{center}\vspace{0.5cm} where $\mu\colon H\to H\otimes_\ka H$ is the comultiplication map. Thus, finite group schemes provide a very natural way of describing iterativity of Moosa-Scanlon operators.

Model theory of $\mathfrak{g}$-fields was analysed by Hoffmann and Kowalski in \cite{HK2}, where they proved that the theory of $\mathfrak{g}$-fields has a model companion, which is moreover simple. Thus, the basic model-theoretic properties of the ``right'' notion of iterativity of $B$-operators are settled.

It seems now natural to ask if we can continue the pursuit of iterativity with $\mathcal{B}$-operators. We claim that it is indeed possible by means of \textit{comonads}. These are certain category-theoretical objects (appearing also in computer science, e. g. in functional programming, see \cite{monads}), which in our context can be seen as ``generalized Hopf algebras'' (see Remark \ref{remark: Hopfish}). 

\begin{definition}
A \textbf{comonad} on a category $\mathcal{C}$ is a comonoid object in the category of endofunctors on $\mathcal{C}$. More concretely, a comonad on $\mathcal{C}$ consists of a functor $F\colon \mathcal{C} \longrightarrow \mathcal{C}$ together with two natural transformations:
\begin{enumerate}
    \item the \textit{counit} $\varepsilon\colon F\longrightarrow \id_\mathcal{C}$,
    \item the \textit{comultiplication} $\mu\colon F\longrightarrow F^2$ (here $F^2$ is the composition of $F$ with itself),
\end{enumerate}
such that the following diagrams commute:
\vspace{0.5cm}\begin{center}
    \begin{tikzcd}[row sep=huge,column sep=huge,every label/.append
style={font=\small}]
F \arrow[r, "\mu"] \arrow[d, "\mu"'] & F^2 \arrow[d, "F\mu"] & F \arrow[rd, "\id"] \arrow[r, "\mu"] \arrow[d, "\mu"'] & F^2 \arrow[d, "F\varepsilon"] \\
F^2 \arrow[r, "\mu F"', shift right] & F^3                   & F^2 \arrow[r, "\varepsilon F"']                        & F                            
\end{tikzcd}
\end{center}\vspace{0.5cm}
\end{definition}

\begin{remark}\label{remark: Hopfish}
We are only interested in the case where $\mathcal{C}=\Alg_\ka$ is the category of $\ka$-algebras and $F = \be$ is a coordinate $\ka$-algebra scheme. In this case comonads can be seen as generalizations of Hopf algebras: if $B$ is a $\ka$-algebra, then a comonoid structure on the functor $-\otimes_\ka B$ is the same as a Hopf algebra structure on $B$ (the first diagram above expresses then the coassociativity of comultiplication and the second expresses the counit condition). Thus, comonads on $\Alg_\ka$ might be seen of as ``nonlinear'' Hopf algebras and are a natural candidate for a way of describing iterativity. Also, since our functors $\be$ are representable (by polynomial rings), a comonad amount to adding some structure to the representing object. This additional structure is the structure of a \textit{plethory}, see \cite{plethystic}.
\end{remark}

\begin{definition}\label{def: mu-iterative}
Let $\mathcal{B}$ be a coordinate $\ka$-algebra scheme and give it the structure of a comonad with comultiplication $\mu$ whose counit is the morphism $\pi \colon \be \to \mathbb{S}_\ka$. We say that a $\mathcal{B}$-operator $\partial \colon R \to \mathcal{B}\left( R \right)$ is \textbf{$\mu$-iterative} if the diagram

\vspace{0.5cm}\begin{center}
    \begin{tikzcd}[row sep=huge,column sep=huge,every label/.append
style={font=\small}]
R \arrow[r, "\partial"] \arrow[d, "\partial"']               & \mathcal{B}\left( R \right) \arrow[d, "\mathcal{B}\left( \partial \right)"] \\
\mathcal{B}\left( R \right) \arrow[r, "\mu_R"', shift right] & \mathcal{B}^2\left( R \right)                       
\end{tikzcd}
\end{center}\vspace{0.5cm}commutes.    
\end{definition}

\begin{example}
If we take $\be = H_\otimes$ where $B$ is a Hopf algebra and $\mu$ is induced by the comultiplication, then a field with a $\mu$-iterative $B$-operator is the same thing as a $\mathfrak{g}$-field in the sense of \cite{HK2}, where $\mathfrak{g}= \spec H$.
\end{example}


\begin{remark}\label{remark: comonads are nice}
As affine schemes $\be$ and $\be^2$ are just affine spaces over $\ka$ of dimension $e$ and $e^2$ respectively. Thus $\mu$ correspond to $e^2$ polynomials $\mu_{i,j}\left( \bar{X} \right)\in \ka\left[ \bar{X} \right]$ in $e$ variables. A $\mathcal{B}$-operator $\left( \partial_1, \dotsc, \partial_e \right) = \partial \colon R \to \mathcal{B}\left( R \right)$ is $\mu$-iterative if and only if for every $i,j\leqslant e$ the equality
$$\left( \partial_i \circ \partial_j \right) \left( x \right) = \mu_{i,j}\left( \partial_1\left( x \right), \dotsc, \partial_e \left( x \right) \right)$$
holds for every $x\in R$. In particular, finitely generated $\mu$-iterative field extension are finitely generated as pure field extensions.
\end{remark}

In a different direction, we can take a coordinate $\ka$-algebra scheme $\be$ and consider fields $K$ with $n$ distinct $\be$-operators $\partial_1,\dotsc \partial_n\colon K\to \be\left( K \right)$ and demand that they all (or just some of them) commute.
In yet another direction, we can consider fields with a $\be$-field $\left( K, \partial \right)$ together with an action of a comonad via $\be$-automorphisms. 

\subsection{Some algebraic properties of \texorpdfstring{$\mathcal{B}_\varphi$}{Bphi}-fields}\label{sec: iterative, sub: algebra}

Motivated by Section \ref{sec: iterative, sub: examples}, we give the following definition.

\begin{definition}\label{def: phi}
An \textbf{iterativity condition} is a universal $\mathcal{L}_\be$-sentence $\varphi$ of the form form $\left( \forall x \right) \theta (x)$ where $\theta ( x )$ is a quantifier-free $\mathcal{L}_\be$-formula in one free variable $x$, such that the following holds:
\begin{enumerate}
    \item for any $\be$-field  $K$ generated as a field by a set $A$ we have that $K\models \varphi$ if and only if $K \models \theta ( a )$ for any $a\in A$,
    \item for any field $K$, the ``zero $\be$-field'' $\left( K , \iota_K \right)$ satisfies $\varphi$.
\end{enumerate}
\end{definition}

\begin{example}\label{example: iterativity conditions}
If $\left( \be , \mu \right)$ is a comonad, then the sentence describing $\mu$-iterative $\be$-fields is $\varphi = ``\left( \forall x \right) \theta (x)"$ where $\theta (x)$ is the conjunction of the equalities appearing in Remark \ref{remark: comonads are nice}. One immediately checks that $\varphi$ is an iterativity condition.
\end{example}

\begin{definition}
Let $\left( K, \partial \right)$ be a $\be$-field and let $\left( V, s \right)$ be a $\be$-variety over $K$. We say that $K$ is a \textbf{$\be_\varphi$-field} it satisfies $\varphi$. We say that a $\be$-variety $\left( V, s \right)$ over a $\bv$-field $\left( K, \partial \right)$ is a \textbf{$\mathcal{B}_\varphi$-variety} if the corresponding $\be$-field $K\left( V \right)$ is a $\be_\varphi$-field.
\end{definition}

Since all examples from Section \ref{sec: iterative, sub: examples} are given in a functorial manner, it is easy to check that they all give rise to iterativity conditions in the above sense.

\begin{remark}\label{remark: coding}
Being a $\mathcal{B}_\varphi$-variety is definable condition, by which mean the following. Fix a $\be_\varphi$-field $K$ and any natural numbers  $n, k, d$. Let us consider all $K$-algebraic sets $V\subseteq \Omega^n$ together with a morphism $s\colon V \to \tau^\partial V$ such that the ideal $I_K\left( V \right)$ is generated by at most $k$ polynomials of degree at most $d$ and that the morphism $s$ is given by polynomials of degree at most $d$. Having fixed that, we can code all the polynomials mentioned in the previous sentence using tuples $\bar{c}$ of of fixed length of elements of $K$. This length depends only on $n, k, d$. Thus we can code $V$ via a tuple of elements of $K$ (in a highly non-unique manner). Then, there is some $\mathcal{L}_\be$-formula $\psi\left( \bar{x} \right)$ depending only on $n, k, d$ such that $\left( V, s \right)$ is a $\be_\varphi$-variety if and only if $K\models \psi\left( \bar{c} \right)$. To see that recall the classical results of van den Dries (see \cite{vdD}) on bounds for polynomial ideals. Namely, for any $n, k, d \in \omega$ there is a number $N$ such that for any field $K$ and any polynomials $h, f_1, \dotsc , f_k \in K\left[ X_1, \dotsc , X_n \right]$ of degree at most $d$ \textbf{if} there are some $g_1\dotsc , g_k\in K\left[ X_1, \dotsc , X_n \right]$ such that
$$h=g_1 f_1 +\dotsc + g_k f_k,$$
\textbf{then} there are such $g_1,\dotsc, g_k$ of degree at most $N$. Using this, one can write down a formula $\chi\left( \bar{x} \right)$ such that $K\models \chi\left( \bar{c} \right)$ if and only if $V$ is $K$-irreducible, the image of $s$ is a subset of $\tau^\partial V$ and that $\left( V, s \right)$ is a $\be$-variety. Since $\varphi$ is universal can be checked on generators (see item (1) in Definition \ref{def: phi}), there is also a formula $\psi\left( \bar{x} \right)$ asserting that moreover $\left( V, s \right)$ is a $\bv$-variety.
\end{remark}

We introduce the following useful class of pairs $\left( \be , \varphi \right)$.

\begin{definition}\label{def: nice}
We say that the pair $\left( \be , \varphi \right)$ is \textbf{nice} if one of the following holds:
\begin{enumerate}
    \item $\left( \be , \varphi \right)$ comes from the action of a comonad on a coordinate $\ka$-algebra scheme $\be_0$ satisfying $\fr  (\ker  \pi_{\be_0}) = 0$.
    \item $\operatorname{char} \ka = 0$ and $\be$ is local and $\varphi$ is trivial (i.e. holds always).
 \end{enumerate}
\end{definition}

\begin{prop}\label{prop: pure assumption}
Assume that $\left( \be , \varphi \right)$ is nice. Let $\left( K, \partial_K \right)\subseteq \left( L, \partial_L \right)$ be $\be_\varphi$-fields, let $a\in L$ be a finite tuple and set $K_1 = K\left( \partial_L a \right)$. Then there is a $\be$-operator $\partial\colon K_1 \to \be\left( K_1 \right)$ such that $\partial\left( a \right) = \partial_L \left( a \right)$, $\partial|_K = \partial_K$ and $\left( K_1, \partial \right)$ is a $\be_\varphi$-field.
\end{prop}
\begin{proof}

Assume we are in item (2) of Definition \ref{def: nice}, then the extension $K\left( a \right) \subseteq K_1$ is separable (since we are in characteristic zero), hence \'{e}tale, thus the $\be$-operator $\partial\colon K\left( a \right) \to \be\left( K_1 \right)$ extends to a $\be$-operator on $K_1$ and since $\varphi$ is trivial, the resulting $\be$-field $K_1$ is a $\bv$-field.

Now, if we are in item (1) of Definition \ref{def: nice}, then we proceed as above and use Remark \ref{remark: comonads are nice}.\end{proof}

Let $\left( K, \partial \right)\subseteq \left( L, \partial \right)$ be $\be_\varphi$-fields and let $a\in L$ be a finite tuple such that $L=K\left( a \right)$. In general there is no well-defined ``$\be$-locus of $a$ over $K$'', as $K\left[ a \right]$ is not necessarily closed under $\partial$. The following lemma says however, that (for some $\be$) we can construct the $\be$-locus at the cost of extending the tuple $a$. 

\begin{lemma}\label{lemma: Blocus exist}
Assume that either $\be$ is local or $\left( \be, \varphi \right)$ is nice. Let $\left( K, \partial \right)\subseteq \left( L, \partial \right)$ be $\be_\varphi$-fields and let $a\in L$ be a finite tuple such that $L=K\left( a \right)$. Then $a$ can be extended to a finite tuple $b$ so that $\partial$ restricts to a $\be$-operator on $K\left[ b \right]$.
\end{lemma}
\begin{proof}
Assume $\be$ is local. Note that we have a $\be$-operator $\partial\colon K\left[ a \right]\to \be\left( K\left[ \partial a \right] \right)$ and since the extension $K\left[ a \right]\subseteq K\left( \partial a \right)$ is \'{e}tale (because it is a localization) this $\be$-operator extends uniquely to a $\be$-operator $\partial'$ on $K\left[ a \right]\subseteq K\left( \partial a \right)$. 
But by uniqueness we must have $\partial' = \partial$ so the tuple $b:=\partial\left( a \right)$ works. If $\be$ is a comonad, then pick $f\in K\left[ a \right]$ such that $\partial\left( a \right) \in \be\left( K\left[ a, 1/f \right] \right)$. By Remark \ref{remark: comonads are nice} $\partial\left( 1/f \right) \in \be\left( K\left[ a, 1/f \right] \right)$, so we can take $b:= \left( a, 1/f \right)$. Finally, if we are in the case (1) of Definition \ref{def: nice}, then we can combine the proof of both cases above
\end{proof}

Amalgamation also transfers easily to the case of $\bv$-fields.

\begin{prop}\label{amal}
Assume $\be$ is local and $\varphi$ is an iterativity condition. Then class of $\be_\varphi$-fields has the amalgamation property in the language $\mathcal{L}_\be^\lambda$ - in other words, any two separable extensions of a $\bv$-field can be amalgamated into a separable extension.
\end{prop}
\begin{proof}
We work inside some large algebraically closed field $\Omega$. Let $L \subseteq M, N$ be separable extensions of $\be_\varphi$-fields. We will first find an amalgam of $M, N$ over $L$ in the class of $\be$-fields. We can replace $L, M, N$ by their separable closures, so without loss of generality assume that $L, M, N$ are separably closed. Since $L$ is separably closed and the extensions $L\subseteq M, N$ are separable, they are in fact regular. Thus their tensor product $M\otimes_K N$ is a domain 
so we may form its field of fractions $F$. 
There is a (unique) $\be$-field structure on $F$ extending the ones on $M$ and $N$, and since the extensions $M, N\subseteq F$ are separable, $F$ is an amalgam of $M, N$ over $K$ considered in the language $\mathcal{L}_\be^\lambda$. Since $M, N$ are $\bv$-fields, item (1) in Definition \ref{def: phi} implies that $F$ is also a $\bv$-field.
\end{proof}

\begin{remark}\label{remark: fin gen amal}
Assume $\be$ is local. If $L \subseteq M, N$ are finitely generated\footnote{As $\be$-extensions or as extensions of pure fields.} separable $\be$-field extensions, then one can take the amalgam of $M,N$ over $L$ to be finitely generated. Indeed, if $K$ is any amalgam, then the $\be$-subfield of $K$ generated by (the images of) $M$ and $N$ has the desired properties.
\end{remark}

Reasoning as in the proof of Proposition \ref{amal} we may deduce the following result.

\begin{lemma}\label{extension by pth roots}
Assume that $\neat$. Then, for any $\bv$-field $\left( K, \partial \right)$ and any $a\in K^\partial$ there is a $\bv$-field structure on $K\left( a^{1/p} \right)$.    
\end{lemma}
In particular, existentially closed $\bv$-fields are \textbf{strict}, i.e. $K^\partial = K^p$.

\section{Elementarity of some variants of existential closedness}\label{sec: general}

Since the pioneering work of Robinson existentially closed models and model companions are indispensable objects in the model theory of algebraic structures. What we want to investigate in this section is the class of $\bv$-fields which are existentially closed, possibly in some restricted class of extensions (e. g. only in regular extensions). This fit into a well-established line of research (see Remark \ref{remark: what we did}). 


In this section we will prove a very general statement about elementarity of some variants of existential closedness (Theorem \ref{mainThm}). This entails and generalizes many results from the literature, as we will see in Section \ref{sec: apl}.

We fix a coordinate $\ka$-algebra scheme $\be$ and an iterativity condition $\varphi$ (see Definition \ref{def: phi}). Let $\ce$ be a class of extensions of $\be_\varphi$-fields.\footnote{In order to avoid confusion: by this we mean that the elements of $\ce$ are pairs $\left( K, L \right)$ where $K \subseteq L$ is an extension of $\be_\varphi$-fields.} The main examples to keep in mind are the class of all extensions, the class of separable extensions and the class of regular extensions (see also Example \ref{example: defclass}).

\begin{definition}
Let $K$ be a $\be_\varphi$-field. We say that a $\be_\varphi$-variety $\left( V, s \right)$ over $K$ is \textbf{of type $\ce$} if the $\bv$-extension $K\subseteq K\left( V \right)$ is in $\ce$.
\end{definition}

From now on we assume that $\ce$ is \textbf{definable} in the following sense: given a $\be_\varphi$-variety $\left( V,s \right)$ over a $\be_\varphi$-field $K$, the property ``$\left( V,s \right)$ is of type $\ce$'' is definable (in the sense of Remark \ref{remark: coding}).

\begin{example}\label{example: defclass}
There are three natural examples of definable classes $\ce$:
\begin{enumerate}
    \item $\ce$ is the class of all $\be_\varphi$-extension. This class is definable directly by Remark \ref{remark: coding}, as $\be_\varphi$-varieties of type $\ce$ are just $\be_\varphi$-varieties.
    \item $\ce$ is the class of all regular extensions. Let $\left( V,s \right)$ be a $\be_\varphi$-variety coded by $\bar{a}$ (in the sense of Remark \ref{remark: coding}). Then $\left( V,s \right)$ is of type $\ce$ if and only if  $V$ is absolutely irreducible, i.e. irreducible over $K^{\alg{}}$. As in Remark \ref{remark: coding}, there is a formula $\psi\left( \bar{x} \right)$ (independent of $V$ and $K$) such that $V$ is irreducible over $K^{\alg{}}$ if and only if $K^{\alg{}}\models \psi\left( \bar{a} \right)$.\footnote{Alternatively, one could use that in $\operatorname{ACF}$ the Morley degree is definable.} Since $\operatorname{ACF}$ eliminates quantifiers, we can assume that $\psi\left( \bar{x} \right)$ is quantifier free, thus $K^{\alg{}}\models \psi\left( \bar{a} \right)$ if an only if $K\models \psi\left( \bar{a} \right)$. Hence $\ce$ is definable.
    \item $\ce$ is the class of all separable extensions. The argument is almost the same as in the previous point, but this time we have to use the theory $\operatorname{SCF}$ in an appropriate language.
\end{enumerate}
\end{example}

Inspired by \cite{SDCF}, we introduce the following notion, which is our central object of interest.
\begin{definition}
We say that a $\be_\varphi$-field $\left( K , \partial \right)$ is \textbf{$\ce$-existentially closed in $\ce$} (or simply \textbf{$\ce$-closed}) if for every $\be_\varphi$-field extension $\left( K , \partial \right)\subseteq \left( L , \partial \right)$ in $\ce$ we have that $\left( K , \partial \right)$ is existentially closed in $\left( L , \partial \right)$ in the language $\mathcal{L}_\be$.
\end{definition}
As we will later many interesting notions from differential algebra and beyond can be interpreted as $\ce$-closedness for appropriate $\ce$. The question we are interested in is the following: when is the property ``being $\ce$-closed'' an elementary property? For the sake brevity, if this property is elementary we will say that \textbf{$\ec$ exists} and denote by $\ec$ the first order theory axiomatising the class of $\ce$-closed $\be_\varphi$-fields. Otherwise we will say that \textbf{$\ec$ does not exist}. If $\ce$ is the class of all extensions of $\bv$-fields, then we will drop $\ce$ from the notation and write $\bvcf$ and if moreover $\varphi$ is trivial we will omit $\varphi$ and write simply $\be\operatorname{CF}$.

\begin{example}
Let us see what does being $\ce$-closed mean for $\ce$ as in Example \ref{example: defclass}.
\begin{enumerate}
    \item If $\ce$ is the class of all $\be_\varphi$-extensions, then $\ec$ exists if and only if the theory of $\be_\varphi$-fields has a model companion in the language $\mathcal{L}_\be$.
    \item Assume that $\be=\left( \ka\left[ X \right] / \left( X^2 \right) \right)_\otimes$, $\varphi$ is trivial (i.e. $\be_\varphi$-fields are differential fields) and $\ce$ is the class of all separable extensions of differential fields. Then $\ec$ exists and has a very nice axiomatization via ``Wood axioms'', as was proved by Ino and Le\'{o}n S\'{a}nchez in \cite{SDCF}. In fact precisely that paper inspired the author of this thesis to investigate existential closedness in restricted classes of extensions.
    \item If $\ce$ is the class of all regular extensions, then being $\ce$-closed is closely related to being \textit{pseudo algebraically closed} in the sense of \cite{PAC}, which we will discuss in Subsection \ref{sec: PAC}.
\end{enumerate}
\end{example}

When checking existential closedness, we will often use the following standard reduction. We skip its proof.

\begin{lemma}\label{lemma: red of formulas}
Let $K\subseteq L$ be an extension of $\be$-fields. Then the following are equivalent:
\begin{enumerate}
    \item $K$ is existentially closed in $L$ in the language $\mathcal{L}_\be$.
    \item For every $K$-polynomials $f_0,\dotsc , f_n$, if $L\models \left( \exists \bar{x} \right) \left( \bigwedge_{i\leqslant n} f_i\left( \partial\bar{x}\right) = 0 \right)$ then $K\models \left( \exists \bar{x} \right) \left( \bigwedge_{i\leqslant n} f_i\left( \partial\bar{x}\right) = 0 \right)$.
\end{enumerate}
\end{lemma}

We introduce the following finiteness condition, which will allow us to give a simpler criterion for $\ce$-closedness (see Lemma \ref{reduction to finitely generated}), which in turn will assure that $\ec$ exists. 

\begin{assumption}\label{ass}
Whenever $\left( K, \partial \right)\subseteq \left( L, \partial \right)$ is in $\ce$ and $a\in L$ is a finite tuple, then there are finite tuple $b$ containing $\partial a$ and a $\be$-operator $\partial'\colon K\left(  b \right) \to \be \left( K\left(  b \right)\right)$ such that $\partial'$ extends the $\be$-operator $\partial |_{K\left( a \right) } \colon K\left( a \right) \to \be \left( K\left( \partial a \right)\right)$ and $\left( K, \partial \right)\subseteq \left( K\left( b \right), \partial' \right)$ is in $\ce$.
\end{assumption}

\begin{remark}\label{remark: many examples}
Let $\ce$ be any of the classes in Example \ref{example: defclass} and let $\left( \be, \varphi \right)$ be a nice pair. Immediately by Proposition \ref{prop: pure assumption} we get that Assumption \ref{ass} holds in this case. This gives a plethora of examples of $\ce, \be, \varphi$ satisfying Assumption \ref{ass}.
\end{remark}

\begin{notation}
We denote by $\ce_{\operatorname{fin}}$ the subclass of $\ce$ consisting of those extensions $K\subseteq L$ in $\ce$ which are finitely generated as pure fields. Note that $\ce_{\operatorname{fin}}$ always satisfies Assumption \ref{ass}, as one can take $b$ to be a tuple of generators of the extension $K\subseteq L$. It is also clearly definable, provided $\ce$ is.
\end{notation}

The meaning of Assumption \ref{ass} is contained in the following result.

\begin{lemma}\label{reduction to finitely generated}
Assume $\ce$ satisfies Assumption \ref{ass}. Then a $\bv$-field $\left( K , \partial \right)$ is $\ce$-closed if and only if it is $\ce_{\operatorname{fin}}$-closed.
\end{lemma}
\begin{proof}
Assume that $\left( K , \partial \right)$ is $\ce_{\operatorname{fin}}$-closed. Let $\phi\left( x\right)$ be a quantifier-free $\mathcal{L}_\be\left( K \right)$-formula such that there is some $\bv$-field extension $\left( K, \partial\right) \subseteq \left( L, \partial_L\right)$ in $\ce$ and some tuple $a\in L$ such that $L\models \phi \left( a \right)$. Using Lemma \ref{lemma: red of formulas} we may assume that $\phi \left( \bar{x} \right)$ is of the form $\theta\left( \partial\bar{x} \right)$ where $\theta\left( \bar{y} \right)$ is a quantifier-free $\mathcal{L}_{\operatorname{rng}}\left( K \right)$-formula. Take $b\in L$ and $\partial'$ as in the conclusion of Assumption \ref{ass}. Then still $K\left( b \right) \models \theta \left( \partial' a \right)$ as $\partial_L \left( a \right) = \partial'\left( a \right)$. Since $\left( K , \partial \right)$ is $\ce_{\operatorname{fin}}$-closed  and $K\subseteq K\left( b \right)$ is a finitely generated extension in $\ce$, we have that there is some tuple $c\in K$ such that $K\models \theta\left( \partial c \right)$, which finishes the proof.
\end{proof}

For technical reasons we also introduce the following assumption, saying that ``$\left( \be, \varphi \right)$ satisfies Lemma \ref{lemma: Blocus exist}''.

\begin{assumption}\label{ass2}
Let $\left( K, \partial \right)\subseteq \left( L, \partial \right)$ be $\be_\varphi$-fields and let $a\in L$ be a finite tuple such that $L=K\left( a \right)$. Then $a$ can be extended to a finite tuple $b$ so that $\partial$ restricts to a $\be$-operator on $K\left[ b \right]$.    
\end{assumption}

\begin{remark}
Let us comment on the nature of the introduced assumptions. Both of the them are finiteness conditions. Assumption \ref{ass} allows us (at least partially) to reduce the study of $\be_\varphi$-fields to the case of $\be_\varphi$-fields which are finitely generated as field over some base $\bv$-field $K$. Assumption \ref{ass2} on the other hand reduces the latter to the study of $\bv$-ring over $K$ which are finitely generated as $K$-algebras, i.e. to the study of $\bv$-varieties (see Proposition \ref{ecfin <=> points}).
\end{remark}

\begin{prop}\label{ecfin <=> points}
Suppose Assumption \ref{ass2} holds. Then, for a $\bv$-field $\left( K , \partial \right)$ the following conditions are equivalent.
\begin{enumerate}
    \item $\left( K , \partial \right)$ is $\ce_{\operatorname{fin}}$-closed.
    \item Every $\bv$-variety over $K$ of type $\ce$ has a Zariski-dense set of $K$-rational $\be$-point.
    \item Every $\bv$-variety over $K$ of type $\ce$ has a $K$-rational $\be$-point.
\end{enumerate}
\end{prop}
\begin{proof}
$\left( 1 \right) \Longrightarrow \left( 2 \right):$ Assume that $\left( K , \partial \right)$ is $\ce_{\operatorname{fin}}$-closed, let $\left( V, s \right)$ be a $\bv$-variety over $K$ of type $\ce$ and let $f$ be a non-zero regular function on $V$. We aim to prove that $\left( V, s \right)$ has a $K$-rational $\be$-point $c$ such that $f\left( c \right) \neq 0$. Let $a\in V$ be a generic point of $V$ over $K$ and $\partial$ be the $\be$-operator on $K\left( a \right) = K\left( V \right)$ corresponding to $s$. Let $\phi\left( x \right)$ be a quantifier-free $\mathcal{L}_\be\left( K \right)$-formula expressing the property
``$x\in \left( V, s \right)^\sharp$ and $f\left( x \right)\neq 0$''.

By Lemma \ref{generic point is sharp} we have that $a\in \left( V, s \right)^\sharp \left( K\left( a \right) \right)$ and $f\left( a \right)$, so that $K\left( V \right) \models \left( \exists x \right)\phi( x )$. Since the extension $K\subseteq K\left( V \right)$ is in $\ce$ we have that $K \models \left( \exists x \right)\phi( x )$, i.e. there exists a $K$-rational $\be$-point of $\left( V, s \right)$ outside from the zero set of $f$. Thus $\left( V, s \right)^\sharp\left( K \right)$ is Zariski-dense in $V$.

$\left( 2 \right) \Longrightarrow \left( 3 \right):$ Obvious.

$\left( 3 \right) \Longrightarrow \left( 1 \right):$ Assume that every $\bv$-variety over $K$ of type $\ce$ has a $K$-rational $\be$-point. Let $\theta\left( x \right)$ be a quantifier-free $\mathcal{L}_{\be}\left( K \right)$-formula such that there is some finitely generated $\bv$-field extension $\left( K , \partial \right)\subseteq \left( L , \partial \right)$ in $\ce$ and some tuple $a\in L$ such that $L\models \theta \left( a \right)$. Our goal is to prove that we can find such a tuple already in $K$. By Lemma \ref{lemma: red of formulas} we can assume that $\theta\left( x \right) = \phi\left( \partial x \right)$ where $\varphi$ is of the form $\bigwedge_{i=0}^n f_i\left( y \right) = 0$ for some $K$-polynomials $f_0,\dotsc, f_n$. Since $L$ is finitely generated we can assume that $L=K\left( a \right)$ (by possibly enlarging the tuple $a$ and adding some dummy variables to $\phi$) and using Assumption \ref{ass2} we can furthermore assume that $\partial$ restricts to a $\be$-operator on $K\left[ a \right]$. Set $V=\locus_K\left( a \right)$ and let $s\colon V \to \tau^\partial V$ be the section of $\pi^\partial_V$ corresponding to the $\be$-operator $\partial\colon K\left[ a \right] \to \be\left( K\left[ a \right]\right)$. In particular $\partial\left( a \right) = s\left( a \right)$. By assumption, we have that $\left( V, s \right)$ has a $K$-rational $\be$-point, say $b\in K$. Thus $\partial\left( b \right) = s\left( b \right)$. Since $a$ satisfies the $\mathcal{L}_\be\left( K \right)$-formula $\phi\left( \partial\left( x \right) \right)$ and $\partial\left( a \right) = s\left( a \right)$, we have that $a$ satisfies the $\mathcal{L}_{\operatorname{ring}}\left( K \right)$-formula $\phi\left( s\left( x \right) \right)$. Since $\phi\left( s\left( x \right) \right)$ is a positive $\mathcal{L}_{\operatorname{ring}}\left( K \right)$-formula and $b\in V=\locus_K\left( a \right)$ we have that $b$ satisfies $\phi\left( s\left( x \right) \right)$ and thus $\phi\left( \partial\left( x \right) \right)$, as desired.
\end{proof}

\textbf{Axioms for $\ec$}

\begin{quote}
$\left( K, \partial \right)$ is a $\be_\varphi$-field such that every $\ce$-variety over $\left( K , \partial \right)$ has a $K$-rational $\be$-point.
\end{quote}

Combining Lemma \ref{reduction to finitely generated}, Proposition \ref{ecfin <=> points} and the ``definability property'' of $\ce$ we get the following result.

\begin{theorem}\label{mainThm}
Suppose that Assumption \ref{ass} and Assumption \ref{ass2} hold. Then, a $\be_\varphi$-field $\left( K, \partial \right)$ is $\ce$-closed if and only if it satisfies the Axioms for $\ec$ written above. In particular, $\ec$ exists.
\end{theorem}

In particular, any nice pair $\left( \be, \varphi \right)$ meets the assumptions of Theorem \ref{mainThm} by Lemma \ref{lemma: Blocus exist} and Remark \ref{remark: many examples}, hence we get the following.

\begin{theorem}\label{mainThm2}
Assume $\left( \be, \varphi \right)$ is nice and $\ce$ is any of the classes in Example \ref{example: defclass}. Then, for a $\be_\varphi$-field $\left( K, \partial \right)$ the following properties are equivalent:
\begin{enumerate}
    \item $\left( K, \partial \right)$ is $\ce$-closed.
    \item Every $\be_\varphi$-variety over $\left( K , \partial \right)$ has a $K$-rational $\be$-point.
\end{enumerate}
Moreover, the latter property is expressible by a scheme of first-order sentences in the language $\mathcal{L}_\be$. In particular, $\ec$ exists. 
\end{theorem}

The rest of this Chapter is devoted to applying Theorem \ref{mainThm2} for various $\be, \varphi, \ce$.

\begin{remark}
We could well ignore any pair $\left( \be, \varphi \right)$ which is not nice, in particular we could refrain from introducing Assumption \ref{ass}. We think however that it is appropriate to introduce them in order to show that our proof work under some abstract structural properties and is not dependent on the particular class of operators involved. Nice pairs are simply a wide class of natural examples to which we can apply our results.
\end{remark}

\section{Applications}\label{sec: apl}
\subsection{Model companions of theories of \texorpdfstring{$\be$}{B}-fields}\label{sec: bcf}
We will use Theorem \ref{mainThm} to prove the existence of a model companion of various theories of $\be_\varphi$-fields, generalizing and unifying many results from the literature. After that we focus on the case when $\be$ satisfies $\neat$. In the latter part of of this subsection we give a different axiomatization of the resulting model companion, in the spirit of the Pierce-Pillay axioms for $\operatorname{DCF}_0$. 

Immediately from Theorem \ref{mainThm} and Proposition \ref{prop: pure assumption} we get the following results.

\begin{theorem}\label{thm: mc}
Assume that $\left( \be, \varphi\right)$ satisfies Assumption \ref{ass} where $\ce$ is the class of all $\bv$-extensions. Then the theory of $\be_\varphi$-fields has a model companion. In particular, this holds for any nice pair $\left( \be, \varphi \right)$.
\end{theorem}

\begin{remark}\label{remark: what we did}
The above theorem unifies and generalizes in one swift motion many existing results about the existence of model companions for theories of fields with operators, mostly in positive characteristic. Below is a list of some of those theories (see also the discussion Section \ref{sec: iterative} and the chart provided in the Introduction). The point is that any of the example below can be described using a nice pair $\left( \be, \varphi \right)$.\footnote{The examples below are precisely the reason why we crafted the definition of a nice pair the way we did.}
\begin{enumerate}
    \item free ``local'' $\mathcal{D}$-ring structures in characteristic zero (see \cite{MS2}),
    \item $B$-operators for $B$ satisfying $\fr\left( \ker \pi_B \right) = 0$ (see \cite{BHKK}; these are precisely the local $B$ for which a model companion exists),
    \item fields with an action of a (fixed) finite group (see \cite{HK1}) or a finite group scheme (see \cite{HK2})
    \item ordinary differential fields with finite group actions (this was done for characteristic zero in \cite{GDCF}, and for positive characteristic in \cite[Theorem 4.36]{PAC}).
\end{enumerate}
What Theorem \ref{mainThm2} does not entail however are partial differential fields in characteristic zero. In fact, the model companion in this case can not be axiomatized using $\bv$-varieties (see Subsection \ref{subsec: sad}).
\end{remark}

Let us now assume that $\neat$ and that $\varphi$ is trivial. We will now give different axioms for $\bcf$, in the spirit of the Pierce-Pillay axioms for $\operatorname{DCF}_0$ given in \cite{PiercePillay}.

\textbf{New Axioms for $\bcf$}

\begin{quote}
For every $K$-varieties $V$ and $W$, if $W\subseteq \tau^\partial V$ and the projection $W\to V$ is separable, then there is some $a\in V\left( K \right)$ such that $\partial_V\left( a \right) \in W\left( K \right)$.
\end{quote}

\begin{theorem}\label{thm: new axioms}
Assume $\be$ satisfies $\neat$. Then a $\be$-field $\left( K, \partial\right)$ is existentially closed if and only if it satisfies the New Axioms.    
\end{theorem}
\begin{proof}
$\left( \Longrightarrow \right)$ Assume $\left( K, \partial\right)$ is existentially closed and take $V, W$ as in the axioms. Let us work inside some big algebraically closed field $\Omega$. Since the projection $W\to V$ is separable, it is by definition dominant. Thus $W=\locus_K \left( a, b \right)$ and $V=\locus_K \left( a \right)$ for some tuples $a, b\in \Omega$. Since $W\subseteq \tau^\partial V$, the is a natural $\be$-operator $\partial\colon K\left[ a \right] \to \be\left( K\left[ a,b \right]\right)$ of the inclusion $K\left[ a \right]\subseteq K\left[ a, b \right]$ such that $\partial\left( a \right) = \left( a,b \right)$. 
The operator $\partial$ extends to a $\be$-operator of the field extension $K\left( a \right) \subseteq K\left( a,b \right)$. This extension is separable so 
$\partial$ extends further to a $\be$-operator on the field $K\left( a,b \right)$. Let $x$ be a tuple of variables of length equal to the length of $a$ and let $\phi \left( x \right)$ be the quantifier-free $\mathcal{L}_\be\left( K \right)$-formula expressing that $\partial_V\left( x \right) \in W$. Since $\phi \left( x \right)$ has a solution in $K\left( a,b \right)\supseteq K$, it has a solution in $K$ by existential closedness. Thus $K$ satisfies the above axioms.

$\left( \Longleftarrow \right)$ Assume $\left( K, \partial\right)$ satisfies the axioms. We will check that $\left( K, \partial\right)$ satisfies the assumptions of Theorem \ref{mainThm2}, i.e. that every $\be$-variety over $K$ has a $\be$-point in $K$. Let $\left( V, s \right)$ be a $\be$-variety over $K$ and define $W := s\left[ V \right] \subseteq \tau^\partial V\subseteq \tau^\partial V$. Note that $W$ is a closed subset of $\tau^\partial V$, since it is equal to the set of all $b\in \tau^\partial V$ such that $s\left( \pi \left( b \right)\right) = b$. Moreover, the projection $W\to V$ is an isomorphism, so in particular it is separable, thus by the New Axioms there is some $a\in V\left( K \right)$ such that $\partial_V\left( a \right) \in W\left( K \right)$ and for this $a$ we have 
$$\partial_V\left( a \right) = s\left( \pi \left(\partial_V\left( a \right)  \right)\right) = s\left( a\right),$$
thus $a$ is a $\be$-point of $\left( V, s \right)$.
\end{proof}

\begin{remark}\label{remark: easy axioms}
Let us point out how Theorem \ref{thm: new axioms} relates to other results in the literature. In \cite[Theorem 3.8]{BHKK} Beyarslan, Hoffmann, Kamensky and Kowalski give an axiomatization of existentially closed $B$-fields with $B$ satisfying $\fr_B \left( \ker \pi_B \right) = 0$. In \cite{beta-op} the author and Kowalski do the same for $\be$-fields with $\be$ satisfying $\neat$ and by taking $\be = B_\otimes$ for $B$ as in the previous sentence one recovers the axiomatization given in \cite{BHKK}. The axioms in \cite[Theorem 4.5]{beta-op} are as follows.

\begin{quote}
For every $K$-varieties $V$ and $W$, if 
\begin{enumerate}
    \item $W\subseteq \tau^\partial V$,
    \item the projection $W\to V$ is dominant,
    \item the projection $E\to W$ is dominant,
\end{enumerate}
then there is some $a\in V\left( K \right)$ such that $\partial_V\left( a \right) \in W\left( K \right)$.
\end{quote}
Here $E$ is a certain algebraic subset of $\tau^\partial \left( W \right)$, defined as the equalizer of certain maps involving $V$ and $W$. Anyway, $E$ looks strange and unnatural, but the assumptions on $V, W$ are \textit{exactly} the conditions assuring that the natural $\be$-operator $\partial_V^W\colon K\left[ V \right] \to \be^\partial\left( K\left[ W \right] \right)$ extends to a $\be$-operator on the field $K\left( W \right)$ (see \cite[Proposition 4.4]{beta-op}).

In our New Axioms instead of introducing $E$ we demand that the projection $W\to V$ is separable. This condition is stronger than ``$E\to W$ is dominant'', thus our Theorem \ref{thm: new axioms} is stronger than \cite[Theorem 4.5]{beta-op} (since we axiomatize the same theory using fewer axioms). At the same time, we achieve an axiomatization of $\be\operatorname{CF}$ faster and with less technicalities than in \cite{BHKK} or \cite{beta-op}. Also, Theorem \ref{thm: new axioms} answers \cite[Question 4]{DerFro}, which asked for an axiomatization of this form in the case of derivations of the Frobenius map (an answer to this question was also achieved in the paper \cite{DerFro2} by the author using a completely different method, see \cite[Remark 3.13]{DerFro2} there; see also Remark \ref{remark: blum}).
\end{remark}

Let us move to model theory. We will prove that for $\left( \be, \varphi \right)$ such that $\neat$ the theory $\bvcf$ eliminates quantifiers after adding the inverse of the Frobenius to the language. \textbf{Throughout the rest of this subsection we assume that $\neat$.}

First we need a certain easy lemma, which is nonetheless very general and useful. Let us introduce a few local definitions. Let $\left( K, \partial \right)$ be a field of characteristic exponent $p$ (i.e. $p$ is the characteristic of $K$ if it is positive and $p=1$ otherwise) with some \textit{operators}, i.e. a tuple (possibly infinite) of unary functions $\partial = \left( \partial_i\colon K\to K \right)_{i \in I}$. We define the \textit{constants} of $\left( K, \partial\right)$ as the set of common zeroes of all $\partial_i$ and denote it by $K^\partial$. We assume that $K^\partial$ is a field. We also assume that every $\partial_i$ is additive and ``$\fr^m$-linear over the constants'', i.e. there is some natural number $m_i$ such that for any $a\in K^\partial,\ x \in K$ we have $\partial_i\left( ax \right) = a^{p^{m_i}} \partial_i\left( x \right)$. Note that under the assumption $\neat$, $\be$-operators fit into this set-up and constants in the sense above are the same as the constant in the sense of $\be$-operators.

Let $\left( K, \partial \right) \subseteq \left( L, \partial' \right)$ be an extension of fields with operators in the above sense, that is, $\partial'|_K = \partial$ and the numbers $m_i$ mentioned above are the same for $K$ and $L$. By abuse of notation, we will use the same symbol $\partial$ for the operators on $K$ and on $L$. The following result was proved in \cite{DerFro2} by the author.

\begin{lemma}
\label{lindisj}
Let $\left( K, \partial \right) \subseteq \left( L, \partial \right)$ be as above. Assume that $p>1$, $L^p\subseteq L^\partial$ and that $K$ is strict, i.e. $K^\partial = K^p$. Then $L^\partial$ and $K$ are linearly disjoint over $K^\partial$.
\end{lemma}
\begin{proof}
Assume the conclusion is not true and take the minimal $n>1$ such that there are some $x_1,\dotsc, x_n\in L^\partial$ linearly dependent over $K$, but linearly independent over $K^\partial$. By the minimality assumption, there are $a_1, \dotsc , a_n\in K\setminus \left\{ 0 \right\}$ such that:

$$a_1x_1+\dotsc + a_nx_n = 0$$
Then for any $i\in I$:

$$0=\partial_i\left( \frac{a_1}{a_n}x_1+\dotsc + \frac{a_{n-1}}{a_n}x_{n-1} +x_n\right)=
\partial_i\left(\frac{a_1}{a_n}\right)x_1^{p^{m_i}}+\dotsc +  \partial_i\left(\frac{a_{n-1}}{a_n}\right)x_{n-1}^{p^{m_i}}$$

If some $\partial_i \left( \frac{a_j}{a_n} \right)$ is nonzero, then $x_1^{p^{m_i}},\dotsc , x_{n-1}^{p^{m_i}}\in L^p\subseteq L^\partial$ are linearly dependent over $K$, so by the minimality assumption on $m$ we get that $x_1^{p^{m_i}},\dotsc , x_{n-1}^{p^{m_i}}$ are linearly dependent over $K^\partial=K^p$, hence $x_1^{p^{m_i-1}}, \dotsc , x_{n-1}^{p^{m_i-1}}$ are linearly dependent over $K$. Repeating this reasoning yields that $x_1,\dotsc, x_{n-1}$ are linearly dependent over $K^p$, contrary to the assumption, that they are independent over $K^\partial=K^p$.

Therefore for any $i$ we have $\partial_i\left( \frac{a_1}{a_n} \right) = \dotsc = \partial_i\left(\frac{a_{n-1}}{a_n}\right)=0$, hence $\frac{a_1}{a_n},\dotsc , \frac{a_{n-1}}{a_n}\in K^\partial$. By the strictness assumption we get that for some $b_1,\dotsc, b_{n-1}\in K\setminus \left\{ 0 \right\}:$
$$a_1=b_1^p a_n, \dotsc , a_{n-1}=b_{n-1}^p a_n,$$
thus
$$0 = a_1x_1+\dotsc + a_nx_n = a_n \left( b_1^p x_1 + \dotsc +b_{n-1}^p x_{n-1} + x_n \right),$$
hence $x_1, \dotsc , x_n$ are linearly dependent over $K^p=K^\partial$, contrary to the assumption.
\end{proof}

From the proof it is clear that in the linear case (i.e. for every $i$ we have $m_i=0$) we have the following

\begin{lemma}
\label{lindisj2}
Let $\left( K, \partial \right) \subseteq \left( L, \partial \right)$ be an extension of fields with operators linear over the constants. Then $L^\partial$ and $K$ are linearly disjoint over $K^\partial$.
\end{lemma}

The strictness assumption in Lemma \ref{lindisj} is necessary (which gives a negative answer to Question 1 in \cite{DerFro}), as shown by the example below.

\begin{example}
Take $K=\mathbb{F}_p\left( X, Y, \lambda, \mu \right), L=\mathbb{F}_p\left( X^{1/p}, Y^{1/p}, \lambda, \mu \right)$ and define a derivation of the Frobenius map on $L$ by setting
$$\partial \left( X^{1/p} \right) = \partial \left( Y^{1/p} \right) = 0, \hspace{0.5cm} \partial \left( \lambda \right) = Y, \hspace{0.5cm} \partial\left( \mu \right) = -X.$$
We will show that $L^\partial$ and $K$ are not linearly disjoint over $K^\partial$. Note that 
$$\partial \left( \lambda X^{1/p} + \mu Y^{1/p} \right) = X\partial \left( \lambda \right)+ Y \partial \left( \mu \right) = 0.$$
Thus, $X^{1/p}, Y^{1/p}, \lambda X^{1/p} + \mu Y^{1/p}$ are elements of $L^\partial$, linearly dependent over $K$. However, they are independent over $K^\partial$: indeed, for any $a, b, c \in K^\partial$, if
$$aX^{1/p}+bY^{1/p} + c\left( \lambda X^{1/p} + \mu Y^{1/p} \right)=0$$
then 
$$\left( a+c\lambda \right) X^{1/p} + \left( b+c\mu \right) Y^{1/p} = 0,$$
but $X^{1/p}$ and $Y^{1/p}$ are linearly independent over $K$, so $a+c\lambda = b+c\mu = 0.$ If $c\neq 0$, then $\lambda = -\frac{a}{c}\in K^\partial$, which is not the case. Thus $c=0$ and therefore $a=b=0$, hence $X^{1/p}, Y^{1/p}, \lambda X^{1/p} + \mu Y^{1/p}$ are linearly independent over $K^\partial$.
\end{example}

Lemma \ref{lindisj} implies that strict $\be$-fields are $\be$-differentially perfect, i.e. any $\be$-field extension is separable (for fields with a derivation of the Frobenius morphism this was asked in Question 2 in \cite{DerFro}).

\begin{lemma}\label{lambda_0 => separable}
Let $K$ be a model of $\bvcf$ and let $K_0$ be subfield of $K$. If $K$ is $\mathcal{L}^{\lambda_0}_\mathcal{B}$-substructure of $K$ then the extension $K_0\subseteq K$ is separable.
\end{lemma}
\begin{proof}
By Remark \ref{extension by pth roots} we have that $K^\partial = K^p$ an assumption we have that $K^p\cap K_0 = K_0^p$, thus $K_0^p=K_0^\partial$. Therefore by Lemma \ref{lindisj} we have that $K^\partial=K^p$ and $K_0$ are linearly disjoint over $K_0^p$, i.e. the extension $K_0\subseteq K$ is separable.
\end{proof}


We need the following technical lemma, which will be also useful in Section \ref{sec: PAC}. The same sort of reduction happens in the middle of the proof of \cite[Theorem 4.34]{PAC}. Here we spell it out in details.

\begin{lemma}\label{lemma: reg}
Let $K\subseteq L$ be an $\mathcal{L}_\be^{\lambda_0}$-extension of strict $\be_\varphi$-fields. Then $K$ is existentially closed in $L$ in the language $\mathcal{L}_\be^{\lambda_0}$ if and only if it is existentially closed in the language $\mathcal{L}_\be$.
\end{lemma}
\begin{proof}
Assume that $K$ is existentially closed in $L$ in the language $\mathcal{L}_\be$ and let $\phi\left( \bar{x} \right)$ be a quantifier-free $\mathcal{L}_\be^{\lambda_0}\left( K \right)$-formula realisable in $L$. Our goal is to prove that $\phi\left( \bar{x} \right)$ is realisable in $K$. We will do this by ``rewriting $\phi$ into the language $\mathcal{L}_\be$'', i.e. we will find a quantifier-free $\mathcal{L}_\be\left( K \right)$-formula $\phi_0\left( \bar{x}' \right)$, where $\bar{x}'$ extends $\bar{x}$, such that
\begin{enumerate}
    \item $\phi_0\left( L \right) \neq \emptyset$ (hence also $\phi_0\left( K \right) \neq \emptyset$),
    \item $L \models \phi_0\left( \bar{x}' \right)\to \phi\left( \bar{x} \right)$.
\end{enumerate}
This immediately implies that $\phi\left( K \right)\neq \emptyset$, as desired.

Using the standard procedures we may assume that $\phi\left( \bar{x} \right) = \psi\left( \partial\left( \bar{x} \right), \lambda_0\left( \bar{x} \right) \right)$, where $\psi\left( \bar{y}, \bar{z} \right)$ is a formula in language of pure fields with parameters from $K$.\footnote{Recall that the first coordinate of the tuple $\partial\left( x \right)$ is $x$.} We can moreover assume that $\psi\left( \bar{y}, \bar{z} \right)$ is of the form $\bigwedge_{i=1}^n f_i\left( \bar{y}, \bar{z} \right)=0$ where $f_1,\dotsc , f_n$ are $K$-polynomials. Say that $\bar{x} = \left( x_1,\dotsc , x_m \right), \bar{z} = \left( z_1,\dotsc , z_m \right)$. Let $\bar{a} =\left( a_1,\dotsc a_m \right)\in L$ be a realization of $\phi \left( \bar{x} \right)$. We partition the set $\left\{ 1, \dotsc , m\right\}$ into three parts $A, B, C$ so that
\begin{itemize}
    \item $A$ consists of those $i$ for which $\lambda_0\left( a_i \right)\neq 0$,
    \item $B$ consists of those $i$ for which $\lambda_0\left( a_i \right) =  0$ and $a_i\neq 0$,
    \item $C$ consists of those $i$ for which $a_i=0$.
\end{itemize}
Define $\bar{x}'$ as the concatenation of $\bar{x}$ and $\bar{z}$. Finally, define $\phi_0\left( \bar{x}'\right)$ as the following formula
$$\psi\left( \bar{x}, \bar{z} \right) \wedge \bigwedge_{i\in A}\left(  z_i^p = x_i\right) \wedge \bigwedge_{i\in B}\left(  z_i=0 \wedge \bigvee_{j=1}^{e-1} \partial_j\left( x_i \right) \neq 0\right)\wedge \bigwedge_{i\in C}\left(  x_i=0\right).$$
Since $L$ is strict, we have that for any nonzero $b\in L$ it holds that $\lambda_0 \left( b \right) = 0$ if and only if $\partial\left( b \right) \neq \iota\left( b \right)$, i.e. for some $j$ we have $\partial_j\left( b \right) \neq 0$. Thus the tuple $\left( \bar{a} , \lambda_0\left( \bar{a} \right) \right)$ witnesses that $\phi_0\left( L \right) \neq \emptyset$. Moreover by the same fact we have that $L\models \phi_0\left( \bar{x}, \bar{z} \right) \to\left( \phi\left( \bar{x} \right) \wedge \bar{z} = \lambda_0\left( \bar{x} \right)\right)$, hence $\phi_0$ is as desired.
\end{proof}
Immediately from the above lemma we get the following result.
\begin{cor}\label{cor: mc with lambda_0}
The theory $\be_\varphi\operatorname{CF}$ is the model companion of the theory of strict $\bv$-fields in the language $\mathcal{L}_\be^{\lambda_0}$.
\end{cor}

Finally, putting all pieces together we arrive at the main result of this Subsection.

\begin{theorem}\label{qe}
The theory $\be_\varphi\operatorname{CF}$ has quantifier elimination in the language $\mathcal{L}^{\lambda_0}_\mathcal{B}$.
\end{theorem}
\begin{proof}
By Corollary \ref{cor: mc with lambda_0} the theory $\be_\varphi\operatorname{CF}$ is the model companion of the theory of strict $\bv$-fields, in particular it is model complete as an $\mathcal{L}_\be^{\lambda_0}$-theory. Since any model complete theory with the amalgamation property admits quantifier elimination, we are done by Proposition \ref{amal}.
\end{proof}

\begin{cor}
Assume that $\ka$ is perfect. Then, the theory $\be_\varphi\operatorname{CF}$ is complete.
\end{cor}
\begin{proof}
By assumption $\left( \ka, \iota_\ka \right)$ is a common $\mathcal{L}^{\lambda_0}_\mathcal{B}$-substructure of all models of $\be_\varphi\operatorname{CF}$, so we are done Theorem \ref{qe}.
\end{proof}

We work with trivial $\varphi$ from now on. Let $\mathfrak{C}\models \bcf$ be a monster model. Considered as a separably closed field e.g. we use the notation $\acl^{\scf}$. On the other hand, $\acl^{\be}$ corresponds to the algebraic closure computed in the $\mathcal{B}$-field $(\mathfrak{C},\partial)$. The proof of the next result is verbatim the same as the proof of \cite[Lemma 4.14]{BHKK}, which deals with the case of $B$-operators.
\begin{lemma}\label{acl}
For any small subset $A$ of $\mathfrak{C}$, we have:
$$\acl^{\be}(A)=\acl^{\scf}(\langle A\rangle_\mathcal{B}), \  \dcl^{\be}(A)=\dcl^{\scf}(\langle A\rangle_\mathcal{B})$$
where $\langle A\rangle_\mathcal{B}$ denotes the $\mathcal{B}$-subfield of $\mathfrak{C}$ generated by $A$.
\end{lemma}
\begin{proof}
We will prove the claim about $\acl^\be$, the proof of the claim about $\dcl^{\be}$ being almost the same. We need to show that $E:=\acl^{\scf}(\langle A\rangle_\be)$ is $\mathcal{B}$-algebraically closed.
Assume not, and take $d\in\acl^{\be}(E)\setminus E$. Let $(K,\partial)\prec(\mathfrak{C},\partial)$ be such that $E\subseteq K$ and such that $K$ also contains the (finite) orbit of $d$ under the action of $\aut^{\be}(\mathfrak{C}/E)$.
There is $f\in\aut^{\scf}(\mathfrak{C}/E)$ such that $f(K)$ is algebraically disjoint from $K$ over $E$. If $f(d)\in K$, then $d\in E$ (since $E$ is $\operatorname{SCF}$-algebraically closed). Therefore $f(d)\not\in K$. Let us denote by $\partial^f$ the $\mathcal{B}$-operator on $f(K)$ which is the transport of $\partial\colon K\to\be\left( K \right)$ via $f$. Arguing as in Theorem \ref{qe}, we get that $Kf(K)\cong \left( f(K)\otimes_E K \right)_0$ and that there is a unique $\mathcal{B}$-operator on this field extending $\partial, \partial^f$. This field of fractions can be embedded (as a $\mathcal{B}$-field) over $K$ into $\mathfrak{C}$. Hence, we can assume that $f(K)$ is an elementary substructure of $(\mathfrak{C},\partial)$ and $f$ is a $\mathcal{B}$-isomorphism. Therefore, $f$ extends to an element of $\aut^{\be}(\mathfrak{C}/E)$. But then we get
$$f(d)\in \left(\aut^{\be}(\mathfrak{C}/E)\cdot d\right)\setminus K,$$
which is a contradiction.
\end{proof}

Using \cite[Theorem 5.8]{fix} one can prove the following
\begin{theorem}\label{thm: stable}
The theory $\bcf$ is stable and not superstable.
\end{theorem}

\begin{lemma}\label{lemma: stationary}
Any complete type over an algebraically closed set is $\ind^{\be}$-stationary
\end{lemma}
\begin{proof}
Let $M \prec \mathfrak{C}$ be a small model, $K$ a small set with $\acl^\be\left( K \right) = K$ and let $a,b\in \mathfrak{C}$ are such that
$$\tp^\be\left( a/ K \right) = \tp^\be\left( b/ K \right), \quad a\ind^\be_{K} M \quad\mbox{ and } \quad b\ind^\be_{K} M.$$
We want to show that also $\tp^\be\left( a/ M \right) = \tp^\be\left( b/ M \right)$.

Let $f\in \aut^\be\left( \mathfrak{C} / K \right)$ be such that $f\left( a \right)$. We see that $f$ constitutes a $\be$-isomorphism between the $\be$-fields $K_a:= \dcl^\be\left( Ka \right)$ and $K_b:= \dcl^\be\left( Kb \right)$. Since $K$ is algebraically closed, the field extension $K\subseteq M$ is regular, so we can do the same juggling as in the proof of Lemma \ref{acl} to arrive at the following: $K_a$ and $M$ are linearly disjoint over $K$ (and the same for $K_b$), $K_a\otimes M$ and $K_b\otimes M$ are $\be$-domains and $f$ lifts naturally to an isomorphism between their fractions fields, i.e to an $\be$-automorphism $\widetilde{f}\colon K_aM\to K_b M$ which agrees with $f$ on $K_a$ (in particular $\tilde{f} \left( a \right) = b$) and is the identity on $M$. It is now enough to extend $\widetilde{f}$ to an automorphism of $\mathfrak{C}$.

Since $K, K_a, K_b$ and $M$ are definably closed, all the extensions $K\subseteq K_a, K_b, M \subseteq\mathfrak{C}$ are separable, so by the linear disjointedness of $K_a$ and $M$ over $K$ (and the same for $K_b$) we get that the extensions $K_aM, K_bM\subseteq \mathfrak{C}$ are separable, in particular they are $\mathcal{L}_\be^{\lambda_0}$ substructures of $\mathfrak{C}$, thus by quantifier elimination (Theorem \ref{qe}) $\widetilde{f}\colon K_aM\to K_b M$ extends to an automorphism of $\mathfrak{C}$, as desired.
\end{proof}

\subsection{PAC structures in the context of \texorpdfstring{$\be$}{B}-fields} \label{sec: PAC}
\textit{Pseudo algebraically closed} (PAC) fields were introduced by Ax in his seminal work \cite{Ax} on pseudofinite fields. Later on Hrushovski in \cite{HrPAC} considered PAC structures in a more general model-theoretic context, more precisely he considered \textit{pseudo algebraically substructures} of strongly minimal structures. In \cite{PillayPAC} Pillay and Polkowska considered PAC substructure in the general context of stable theories and so did recently Hoffmann and Kowalski in \cite{PAC}.\footnote{There is a lot more work on PAC structures and people we should mention, see the introduction to \cite{PAC}.} We will base on the latter work in this section. Let $T$ is be a stable theory and let $\mathfrak{C}$ be a monster model of $T$. 

\begin{definition}[Definition 2.3. in \cite{PAC}]\label{def: PAC}
We say that a small substructure $F\subset \mathfrak{C}$ is \textbf{pseudo algebraically closed} (or $T$-$\operatorname{PAC}$) if and every stationary type over $F$ is finitely satisfiable in $F$.
\end{definition}

For $T=\operatorname{ACF}$ it is an easy exercise to prove that pseudo algebraically closed substructures are precisely perfect pseudo algebraically closed fields. Moreover, in $\operatorname{ACF}$ stationary types correspond to absolutely irreducible varieties, so Definition \ref{def: PAC} generalizes the ``geometric'' definition of being pseudo algebraically closed, i.e. ``PAC = absolutely irreducible varieties have points''. The point is that one can view this notion also as a variant of existential closedness, as explained below in Fact \ref{fact: regcl}.

\begin{definition}[Definition 2.3. in \cite{PAC}]\label{def: reg}
We say that an extension of small substructures $F\subseteq K$ is \textbf{regular} if $\dcl\left( K \right)\cap \acl\left( F \right) = F$.
\end{definition}

Note that in particular, if the extension $F\subseteq K$ is regular then $F$ is definably closed.

\begin{example}\label{example: SCF-PAC}
If $L=\mathcal{L}_{\operatorname{rng}}^\lambda$ and $T=\scf_{p,e}$ (where $e$ is finite or not), then a regular extension in the above sense is the same as regular extension in the field theoretic sense (see Fact 4.17 and Fact 4.20 in \cite{PAC}).
\end{example}

\begin{fact}[Remark 2.17 in \cite{PAC}]\label{fact: regcl}
Assume the following conditions hold:
\begin{enumerate}
    \item $T$ has quantifier elimination.
    \item $T$ codes finite tuples (i.e. eliminates finite imaginaries).
    \item $T$ is stable and types over algebraically closed sets are stationary.
\end{enumerate}
Then, a small subset $F$ is $T$-$\operatorname{PAC}$ if and only if $F=\dcl ( F )$ and $F$ is \textbf{regularly closed}, i.e. existentially closed in every regular extension.
\end{fact}

Specializing to $T=\operatorname{ACF}$ in Fact \ref{fact: regcl} we recover the well-known fact that PAC fields are exactly the fields which are existentially closed in any (field-theoretically) regular extension.

\begin{remark}\label{remark: fin eq}
Any theory of a field (with possibly some additional structure) codes finite tuples - given a tuple $\bar{a} = \left( a_1,\dotsc , a_n \right)$ the tuple $\left( s_1\left( \bar{a} \right) , \dotsc , s_n\left( \bar{a} \right) \right)$ is a code for $\bar{a}$, where $s_1,\dotsc , s_n$ are the elementary symmetric polynomials in $n$ variables.
\end{remark}

Let us apply the above notions in the case of $\be_\varphi$-fields. Fix a coordinate $\ka$-scheme $\be$ satisfying $\neat$ and an iterativity condition $\varphi$. From now on let $T$ be the theory $\be_\varphi\operatorname{CF}$ considered in the language $\mathcal{L}_\be^{\lambda_0}$ and let $\mathfrak{C}$ be a monster model of $T$. Note that $T$ meets the assumptions of Fact \ref{fact: regcl}:  The theory $T$ is stable by Theorem \ref{thm: stable}, stationary of types over algebraically closed sets is the content of Lemma \ref{lemma: stationary}, by the choice of the language it eliminates quantifiers thanks to Theorem \ref{qe} and Remark \ref{remark: fin eq} implies $T$ codes finite tuples.

Let us write $\bpac$ instead of $\bv\operatorname{CF}$-$\operatorname{PAC}$. We aim to prove that $\bpac$ is an elementary property. Fact \ref{fact: regcl} tells us that $\bpac$ is an instance of $\ce$-closedness for an appropriate $\ce$, so this fits perfectly into the set-up of Section \ref{sec: general}. The next few result technical lemmas will show this even more concretely, namely we will prove the following (see Lemma \ref{lemma: PAC final reduction}): a $\bv$-field $K$ is $\bpac$ if and only if $K$ is strict and $\ce_{\operatorname{reg}}$-closed where $\ce_{\operatorname{reg}}$ denotes the class of all $\bv$-field extensions which are regular as extensions of pure fields.

Until the end of this Section, all $\be_\varphi$-field considered are small subsets of $\mathfrak{C}$.


\begin{lemma}\label{lemma: reg is pure}
An extension of small $\be_\varphi$-fields $F\subseteq K$ is regular in the sense of Definition \ref{def: reg} if and only if $F\subseteq K$ is regular as a pure field extension.   
\end{lemma}
\begin{proof}
By Lemma \ref{acl} we have
$$\dcl^{\be}\left( K \right)\cap \acl^{\be}\left( F \right) = \dcl^{\scf}\left( K \right)\cap \acl^{\scf}\left( F \right),$$
thus we are done by 
\ref{example: SCF-PAC}.
\end{proof}

\begin{lemma}\label{lemma: PAC final reduction}
Let $K$ be a $\be_\varphi$-field. Then $K$ is $\bpac$ if and only if $K$ is strict and $\ce_{\operatorname{reg}}$-closed.    
\end{lemma}
\begin{proof}
By Fact \ref{fact: regcl} $K$ is $\bpac$ if and only if $K=\operatorname{dcl} (K)$ and regularly closed. 
\end{proof}
Finally, from Theorem \ref{mainThm2}, Example \ref{example: defclass} and Lemma \ref{lemma: PAC final reduction} we immediately get the following result.

\begin{theorem}\label{thm: PAC}
For a $\be_\varphi$-field $\left( K , \partial \right)$ the following are equivalent:
\begin{enumerate}
    \item $\left( K , \partial \right)$ is $\bpac$.
    \item $K$ is strict and every absolutely irreducible $\be_\varphi$-variety over $\left( K , \partial \right)$ has a $K$-rational $\be$-point.
\end{enumerate}
In particular, being $\bpac$ is an elementary property.
\end{theorem}

\begin{remark}\label{remark: PAC solved}
Theorem \ref{thm: PAC} applies in particular to the theory $\operatorname{DCF}_{p}$ for $p>0$, which was tackled by Hoffmann-Kowalski in \cite{PAC}. Their axiomatization of $\operatorname{DCF}_{p, 0}-\operatorname{PAC}$ is a bit complicated (see the paragraph above \cite[Theorem 4.34]{PAC}). The source of this complication is twofold: 
\begin{itemize}
    \item There is a certain equalizer variety $E$ involved (see also Remark \ref{remark: easy axioms}).
    \item The axioms use a certain notion of ``admissible tuples''. It is not at all obvious that this notion is expressible in first-order logic and in fact the proof that it is uses a slight enhancement of a theorem by Tamagawa (see \cite[Proposition 11.4.1]{field_arithmetic}).
\end{itemize}
Our axioms are much more transparent and also the proof that they work is much easier. Additionally, Theorem \ref{thm: PAC} goes well beyond the case of $\operatorname{DCF}_{p}$
\end{remark}

\begin{remark}
$\ce_{\operatorname{reg}}$-closed $\bv$-fields are also interesting for other classes of $\be$ than considered in this Section. For example, let $G$ be a finitely generated group and let $\left( \be, \varphi \right)$ be the pair describing $G$-fields. Then, $\ce_{\operatorname{reg}}$-closed $\bv$-fields are (by definition) the same as \textit{pseudo existentially closed} $G$-fields (see \cite{torsion}). They implicitly appear in Hrushovski's proof that the theory of fields with two commuting automorphisms does not have a model companion, since what is really proved there is the much stronger statement that there is no $\aleph_0$-saturated pseudo existentially closed $\left(\mathbb{Z}\times\mathbb{Z}\right)$-field. We send the reader to \cite[Section 6.2]{torsion} for more details. In any case, using Theorem \ref{mainThm2} we get that for a finite group $G$ the class of pseudo existentially closed $G$-fields is elementary.
\end{remark}

\subsection{Separably \texorpdfstring{$\be_\varphi$}{Bphi}-closed fields}

In \cite{SDCF} Ino and León Sánchez considered separably differentially closed field, i.e. differential field which are closed in separable differential field extensions (here separable is in the sense of pure field). Among other things, they prove that this class of field is elementary and give a very nice axiomatization of its theory $\operatorname{SDCF}_p$, similar to the axiomatization of $\operatorname{DCF}_0$ by Blum (see \cite{blum}) or the axiomatization of $\operatorname{DCF}_0$ by Wood (see \cite{Wo1}). They also give a geometric axioms for $\operatorname{SDCF}_p$. 

This clearly fits into our context. Fix some coordinate $\ka$-algebra scheme $\be$ and an iterativity condition $\varphi$. Let $\ce_{\operatorname{sep}}$ be the class of all separable (in the field-theoretic sense) extensions of $\bv$-fields. Let us say that a $\bv$-field $\left( K, \partial \right)$ is \textbf{separably $\bv$-closed} if it is $\ce_{\operatorname{sep}}$-closed. Immediately from Theorem \ref{mainThm2} we get the following.

\begin{cor}\label{cor: sbcf}
Assume that $\left( \be, \varphi \right)$ is nice. Then, for a $\bv$-field $\left( K, \partial \right)$ the following are equivalent:
\begin{enumerate}
    \item $\left( K, \partial \right)$ is separably $\bv$-closed.
    \item Every separable $\bv$-variety over $\left( K, \partial \right)$ has a $K$-rational $\bv$-point.
\end{enumerate}
In particular, there is an $\mathcal{L}_\be$-theory $\operatorname{S}\bv\operatorname{CF}$ whose models are precisely separably $\bv$-closed $\bv$-fields.
\end{cor}

\begin{remark}\label{remark: blum}
In particular setting $\be = \ka\left[ \varepsilon\right]_\otimes$ where $\ka\left[ \varepsilon\right] = \ka\left[ X\right] /\left( X^2 \right)$ we get a new axiomatization of the theory $\operatorname{SDCF}_p$ considered in \cite{SDCF}. There are two axiomatizations given in \cite{SDCF}. One of them roughly corresponds to our New Axioms for $\be\operatorname{CF}$ 
. The second one is similar to the axiomatization of $\operatorname{DCF}_0$ given by Blum or the axiomatization of $\operatorname{DCF}_p$ (where $p>0$ is prime) given by Wood, and it speaks about the solvability of certain differential equations in one variable. An analogous results for derivations of the Frobenius map was proven by the author in \cite[Theorem 3.12]{DerFro2}.
\end{remark}

\subsection{Largeness}
Recall that a field $K$ is called \textbf{large} if every $K$-variety defined with a smooth $K$-rational point has a Zariski-dense set of $K$-rational points. In particular, the class of large fields includes pseudo algebraically closed fields (hence also separably and algebraically closed fields) and real closed fields. In general, large fields are in some sense the widest class of ``tame fields'' and they have many remarkable properties (e. g. the regular inverse Galois problem is solvable over large fields). Moreover being large is an elementary property in the language of rings. We send the reader to \cite{Pop} for a survey on large fields.

In \cite{diffLarge} León Sánchez and Tressl introduce a differential counterpart of large fields and among many other thing proved that the class of differentially large fields is elementary. We want to generalize this result to the case of $\bv$-fields. Let $\ce_{\operatorname{ec}}$ be the class of all $\bv$-extensions $K\subseteq L$ such that $K$ is existentially closed in $L$ as a pure field.
\begin{definition}
Let $\left( K, \partial \right)$ be a $\be_\varphi$-field.  We say that $K$ is \textbf{$\be_\varphi$-large} if it is large as a pure field and $\ce_{\operatorname{ec}}$-closed.
\end{definition}
Specializing to $\left( \be, \varphi \right)$ describing partial differential fields, we recover the definition given in \cite{diffLarge}.

We aim to show that for a nice pair $\left( \be, \varphi \right)$ being $\be_\varphi$-large is an elementary property. For partial differential fields in characteristic zero this was done in \cite{diffLarge} and for ordinary differential fields of positive characteristic in \cite{ordLarge}. 

The following lemma is well-known, see e.g \cite[Fact 2.3]{Pop}.

\begin{lemma}\label{lemma: ec dense}
Let $K$ be a field and $V$ a $K$-variety. Then, the following are equivalent:
\begin{enumerate}
    \item $K$ is existentially closed in $K\left( V \right)$.
    \item $V\left( K \right)$ is Zariski-dense in $V$.
\end{enumerate}
\end{lemma}

\begin{cor}\label{large}
Assume that $\left( \be, \varphi \right)$ is nice. Then, for a $\bv$-field $\left( K, \partial \right)$ the following are equivalent:
\begin{enumerate}
    \item $\left( K, \partial \right)$ is $\bv$-large.
    \item $K$ is large and every $\bv$-variety over $\left( K, \partial \right)$ with a smooth $K$-rational point has a $K$-rational $\bv$-point.
\end{enumerate}
Moreover, being $\bv$-large is an elementary property.
\end{cor}
\begin{proof}
$(1) \Longrightarrow (2)$ Assume $(1)$ and let $\left( V, s \right)$ be a $\bv$-variety such that $V$ has a smooth $K$-rational point. Since $K$ is large, $V$ has in fact a Zariski-dense set of such points, thus by Lemma \ref{lemma: ec dense} we have that $K$ is existentially closed (as a pure field) in $K\left( V \right)$. Since $K$ is $\ce_{\operatorname{ec}}$-closed, $K$ is existentially closed in $K\left( V \right)$ as a $\bv$-field. By Lemma \ref{generic point is sharp} there is a $K\left( V \right)$-rational $\bv$-point of $\left( V, s \right)$, thus by existential closedness there is already a $K$-rational one.

$(2) \Longrightarrow (1)$ Assume $(2)$. Since $\left( \be, \varphi \right)$ is nice, by Lemma \ref{reduction to finitely generated} and Proposition \ref{prop: pure assumption} we have to check only that $K$ is $\left( \ce_{\operatorname{ec}} \right)_{\operatorname{fin}}$-closed. But this follows from $(2)$ using Lemma \ref{lemma: ec dense} as in the previous implication.

The moreover claim follows since being large is an elementary property and ``$V$ has smooth $K$-rational point'' is a definable condition in the sense of Remark \ref{remark: coding}.
\end{proof}





\subsection{The theory \texorpdfstring{$\operatorname{DCF}_{0, m}^{\operatorname{fin}}$}{DCFfin}}\label{subsec: sad}
Let us note a peculiar example which pops out naturally from our work. For simplicity of the exposition, let us work with fields of characteristic zero with $m\ge 2$ commuting derivations, which we will also call \textbf{partial differential fields} (suppressing $m$ from the nomenclature). Denote by $\operatorname{DF}_{0, m}$ the theory of such fields in the language $\mathcal{L}_D = \mathcal{L}_{\operatorname{rng}} \cup \left\{ \partial_1, \dotsc , \partial_m \right\}$. McGrail proved in \cite{mcgrail} that $\operatorname{DF}_{0, m}$ has a model companion $\operatorname{DCF}_{0, m}$. Also, by Theorem \ref{mainThm} we get that there is a theory $\operatorname{DCF}_{0, m}^{\operatorname{fin}}$ whose models are exactly those $K\models\operatorname{DF}_{0, m}$ which are existentially closed in any $L\models \operatorname{DF}_{0, m}$ such that $K\subseteq L$ is finitely generated as a pure field extension. There is a natural question whether $\operatorname{DCF}_{0, m} = \operatorname{DCF}_{0, m}^{\operatorname{fin}}$. Note that by Theorem \ref{mainThm2} this is true for $m=1$ or for $p$ instead of $0$.

Unfortunately the answer is no. Consider the field $K=\mathbb{Q}\left( x_1,\dotsc x_m \right)$ with the obvious derivations. By \cite[Theorem 2]{notfin} the (consistent) equation
$$\partial_1\left( f \right) = \left( 1 - \frac{x_1}{x_2} \right) \partial_2 \left( f \right) +1$$
has no solution $f$ such that $K\left\langle f\right\rangle$ has finite transcendence degree over $K$, where $K\left\langle f\right\rangle$ is the partial differential field generated over $K$ by $f$.

Using this one can construct a model $M$ of $\operatorname{DCF}_{0, m}^{\operatorname{fin}}$ which is not a model of $\operatorname{DCF}_{0, m}$. We will invoke Fraïssé theory for this purpose, though one can construct $M$ directly, similarly to how one construct an existentially closed model of an inductive theory, although some care is needed. 

Let us recall some classical material about Fraïssé theory. We refer to Fraïssé's original paper \cite{Fraisse}.

\begin{definition}
Let $\mathcal{L}$ be a countable language. We say that a countable $\mathcal{L}$-structure $M$ is \textbf{ultrahomogeneous} if any isomorphism between finitely generated substructures of $M$ extends to an automorphism of $M$.
\end{definition}

\begin{definition}
Let $M$ be an $\mathcal{L}$-structure. The \textbf{age of $M$} is the class $\operatorname{age}\left( M \right)$ of all finitely generated $\mathcal{L}$-structures which embed into $M$. Equivalently, it is (the closure under isomorphic images of) the class of all finitely generated $\mathcal{L}$-substructures of $M$.
\end{definition}

\begin{definition}
Let $\mathcal{C}$ be a class of finitely generated $\mathcal{L}$-structures. If
\begin{enumerate}
    \item $\mathcal{C}$ is closed under isomorphisms and finitely generated substructures,
    \item $\mathcal{C}$ has only countably many members up to isomorphism,
    \item $\mathcal{C}$ has the joint embedding property,
    \item $\mathcal{C}$ has the amalgamation property,
\end{enumerate}
the we say that $\mathcal{C}$ is a \textbf{Fraïssé class}.
\end{definition}

The following is the celebrated Fraïssé theorem.

\begin{fact}
Let $\mathcal{C}$ be a Fraïssé class. Then there is a unique (up to isomorphism) countable ultrahomogeneous structure $M$ whose age is equal $\mathcal{C}$ called the \textbf{Fraïssé limit} of $\mathcal{C}$. 
\end{fact}

For our purposes, we work in the language $\mathcal{L}_{D, \operatorname{inv}}=\mathcal{L}_{D} \cup \left\{ \operatorname{inv} \right\}$ where ``$\operatorname{inv}$'' is a unary function symbol. Any partial differential field $K$ is naturally an $\mathcal{L}_{D, \operatorname{inv}}$-structure, where we interpret $\operatorname{inv}$ as follows:
$$\operatorname{inv}^K( x ) = 
\begin{cases}
0& \mbox{if } x=0 \\
\frac{1}{x} & \mbox{otherwise}
\end{cases}.$$
Of course, any extension of partial differential fields is also an $\mathcal{L}_{D, \operatorname{inv}}$-extension. We define $\mathcal{C}$ as the class of all partial differential fields which are finitely generated (as pure fields over $\mathbb{Q}$).

\begin{lemma}
The class $\mathcal{C}$ is a Fraïssé class.
\end{lemma}
\begin{proof}
Clearly $\mathcal{C}$ is closed under finitely generated substructures and one easily sees that $\mathcal{C}$ has only countably many objects up to isomorphism. By Proposition \ref{amal} and Remark \ref{remark: fin gen amal} the class $\mathcal{C}$ has the amalgamation property and since $\mathcal{C}$ has an initial object (namely $\mathbb{Q}$ with $m$ trivial derivations) it has also the joint embedding property.
\end{proof}
Denote the Fraïssé limit of $\mathcal{C}$ by $M$. Note that $M$ is a field, since any of its finitely generated substructure is one.
\begin{prop}
$M$ is a model of $\operatorname{DCF}_{0, m}^{\operatorname{fin}}$ which is not a model of $\operatorname{DCF}_{0, m}$.
\end{prop}
\begin{proof}
We will first prove the latter claim. Since the finitely generated substructures of $M$ are (up to isomorphism) precisely all models of $\operatorname{DF}_{0, m}$, we see that the equation
$$\partial_1\left( f \right) = \left( 1 - \frac{x_1}{x_2} \right) \partial_2 \left( f \right) +1$$
has no solution in $M$ (see the beginning of this section). On the other hand, this equation has a solution in some partial differential field $N$, hence (by taking an amalgam of $M$ and $N$) also in an extension of $M$. Thus $M$ is not existentially closed as a partial differential field, hence $M$ is not a model of $\operatorname{DCF}_{0, m}$.

As for the former claim, let $M\subset N$ be an extension of partial differential fields which is finitely generated as an extension of pure fields. We want to prove that $M$ is existentially closed in $N$ in the language $\mathcal{L}_D$. Let $\phi\left( \bar{x}, \bar{y} \right)$ be a quantifier-free $\mathcal{L}_D$-formula and let ${a}\in N$ and ${b}\in M$ be tuples such that $N\models \phi\left( {a}, {b} \right)$. Let $N_0$ be the partial differential field generated by ${a}$ and ${b}$. Then $N_0\in \mathcal{C}$, thus there is an embedding $f\colon N_0\to M$. Set ${b}'=f( {b} )$. Then, $f$ restricts to an isomorphism between the substructures of $M$ generated by ${b}$ and ${b}'$, hence by ultrahomogeneity there is an automorphism $\sigma$ of $M$ such that $\sigma\left( {b} \right) = {b}'$. Since $N_0 \models \phi\left( {a}, {b} \right)$ and $f$ is an embedding we have that $M\models \phi\left( f\left( {a} \right), {b}' \right)$. Applying $\sigma^{-1}$ yields $M\models \phi\left( (\sigma^{-1}\circ f)\left( {a} \right), {b} \right)$, thus $\phi ( \bar{x}, {b} )$ is satisfiable in $M$. Therefore $M$ is existentially closed in $N$ and thus $M$ is a model of $\operatorname{DCF}_{0, m}^{\operatorname{fin}}$.  
\end{proof}

In more fancy terms, models of $\operatorname{DCF}_{0, m}^{\operatorname{fin}}$ are differential fields $K$ (inside some monster model of $\operatorname{DCF}_{0, m}$) with the property that every finite dimensional type over $K$ is finitely satisfiable in $K$.

\bibliographystyle{plain}
\bibliography{biblio}

\end{document}